\documentclass[11pt]{article}
\usepackage{amsmath,amssymb,amsthm,amsfonts, fullpage, setspace}

\usepackage{eucal}
\usepackage{subeqnarray,bm}
 \usepackage[utf8]{inputenc}
\usepackage{graphicx}[final]
\usepackage{psfrag} %only include if using pictures
\usepackage{color,bm}
\usepackage{pdfpages}
\usepackage{epstopdf}
\usepackage[colorlinks]{hyperref}
\usepackage{hyperref}
\hypersetup{colorlinks=red,citecolor=blue,pdfstartview=FitH, pdfpagemode=None}
\usepackage{multirow}
\usepackage[bottom]{footmisc}
\usepackage{ifthen} %only include if using conditional package
\usepackage{sectsty}
\usepackage{subfig}
\usepackage{float}
\usepackage[english]{babel}
\usepackage[nottoc]{tocbibind}
\usepackage{amsfonts}
\usepackage{amssymb}
\usepackage{amsmath, times}
\usepackage{amscd, bm}
\usepackage{graphics}
\usepackage{setspace}
\usepackage{wrapfig}
\usepackage{animate}
\usepackage{scrextend}
\usepackage{lipsum}
\usepackage{fancybox}
\usepackage{listings}
\usepackage{inconsolata}
\usepackage{textpos}
\usepackage{pgf}
\usepackage{booktabs}
\usepackage{latexsym, amsthm}
\usepackage{graphicx}
\usepackage{verbatim}
\usepackage{url}
\usepackage{xcolor,soul,color, colortbl}
\sethlcolor{yellow}
\usepackage[all]{xy}
\usepackage{multirow}
\usepackage{tikz}
\usepackage{fancybox}

 \DeclareGraphicsExtensions{.eps,.png,.pdf}
\usepackage{mwe}
\makeatletter
\def\maxwidth{ %
  \ifdim\Gin@nat@width>\linewidth
    \linewidth
  \else
    \Gin@nat@width
  \fi
}
\makeatother

\definecolor{fgcolor}{rgb}{0.345, 0.345, 0.345}

\allsectionsfont{\normalsize\bfseries}

\usepackage{framed}
\makeatletter
 {\par\unskip\endMakeFramed%
 \at@end@of@kframe}
\makeatother

\definecolor{shadecolor}{rgb}{.97, .97, .97}
\definecolor{messagecolor}{rgb}{0, 0, 0}
\definecolor{warningcolor}{rgb}{1, 0, 1}
\definecolor{errorcolor}{rgb}{1, 0, 0}
\newtheorem{theom}{Theorem}[section]
\newtheorem{thm}{Theorem}
\newtheorem{defn}[thm]{Definition}
\newtheorem{prop}[theom]{Proposition}

\newtheorem{Rem}[theom]{Remark}

\newcommand{\ds}{\displaystyle}

\newcommand{\norm}[1]{\left\Vert#1\right\Vert}
\newcommand{\abs}[1]{\left\vert#1\right\vert}
\newcommand{\set}[1]{\left\{#1\right\}}
\newcommand{\rb}[1]{\left(#1\right)}
\newcommand{\sbb}[1]{\left[#1\right]}

\newcommand{\R}{\mathbb{R}}

\newcommand{\C}{\mathbb{C}}

\newcommand{\Z}{\mathbb{Z}}

\newcommand{\Xint}{\int_{\mathbb{X}}}
\newcommand{\XXint}{\int_{\mathbb{X}^2}}

\def\Xint#1{\mathchoice
{\XXint\displaystyle\textstyle{#1}}%
{\XXint\textstyle\scriptstyle{#1}}%
{\XXint\scriptstyle\scriptscriptstyle{#1}}%
{\XXint\scriptscriptstyle\scriptscriptstyle{#1}}%
\!\int}
\def\XXint#1#2#3{{\setbox0=\hbox{$#1{#2#3}{\int}$ }
\vcenter{\hbox{$#2#3$ }}\kern-.6\wd0}}

\def\dashint{\Xint-}
\newcommand{\aint}{\dashint}

\newcommand{\DD}{\mathbb D}
\newcommand{\TT}{\mathbb T}
\newcommand{\HH}{\mathbb H}
\newcommand{\N}{\mathbb N}

\numberwithin{equation}{section}

\usepackage{calrsfs}
\allsectionsfont{\normalsize\bfseries}

\numberwithin{equation}{section}
\allowdisplaybreaks

\begin{document}
\title{On the  equivalence between Weak BMO and the space of derivatives of the Zygmund class}
\author{
Eddy Kwessi\footnote{Department of Mathematics, Trinity University, San Antonio, TX, USA, ekwessi@trinity.edu}
%\quad  \quad Tim Myers \footnote{Department of Mathematics, Howard University , tmyers@howard.edu} \quad \quad Douglas Mupasiri \footnote{Department of Mathematics, University of Northern Iowa , douglas.mupasiri@uni.edu } \quad \quad Tepper Gill \footnote{Department of Mathematics, Howard University , tgill@howard.edu}\quad \quad Geraldo de Souza \footnote{Department of Mathematics, Auburn University , desougsl@auburn.edu}\\ 
%Trinity University\quad University of Alabama at Birmingham 
 }

\date{}
\maketitle

%\chapter*{}
\begin{abstract}

In this paper, we will discuss the  space of functions of weak bounded mean oscillation.  In particular, we will show that this space is  the dual space of the special atom space, whose dual space was already known to be the space of derivative of functions (in the sense of distribution) belonging to the Zygmund class of functions. We show in particular that this  proves that the Hardy space $H^1$ strictly contains the space special atom space.

%  show that this relationship between the Zygmund $\Lambda_*(I)$ class  and  $BMO^w(I)$ extends further to their analytic characterizations $\Lambda_*(\DD)$ and $BMO^w(\DD)$ on the unit disc $\DD$.

%Human brain contains a large number of neurons that often evolve in large neural networks representing  groups of neural populations where each element interacts under excitement impulses with other elements. Often, systems of continuous differential equations are used to model these ensembles of neurons. In the presence of data however, discrete models are  preferred and it has been well documented that without proper care, discrete and continuous models do not always yield the same dynamics. The method proposed by \cite{Mickens2005, Mickens2015} and later additions aim to address the discrepancies between continuous and discrete models. 
%One  model often used in neuroscience to represent ensembles of neurons  is the FitzHugh-Nagumo system. This system consists of two  ordinary differential equations linking an activator and an inhibitor, and represents the excitability of the neural network. 
% In this paper, we propose a nearly exact  discretization scheme  for the FitzhHugh-Nagumo model. We prove  that the scheme  preserves qualitatively and quantitatively  the dynamics and features of the original continuous system. 

 \end{abstract}

 \vspace{0.25cm} 


\section{Introduction}
The space of functions of bounded means oscillation has taken central stage in the mathematical literature after the work of Charles Fefferman \cite{Fefferman1971} where he showed that it is the dual space of the real Hardy space $\HH^1$, a long sought after result. Right after, Ronald Coifman \cite {Coifman1974} showed this result  using a different method. Essentially, he showed that $\HH^1$ has an atomic decomposition.  De Souza \cite{DeSouza1980} showed there is a subset $B^1$ of $\HH^1$, formed by special atoms that is contained in $\HH^1$.  This space $B^1$ has the particularity that it contains some  functions whose Fourier series diverge. The question of whether $B^1$ is equivalent to $\HH^1$ was never truly answered explicitly but it was always suspected that the inclusion $B^1\subset \HH^1$ was strict, that is, there must be at least one function in $\HH^1$ that is not in $B^1$. However, such a function had never being constructed nor given. Since the dual space $(\HH^1)^*$ of $\HH^1$ is $BMO$ and $B^1\subset \HH^1$, it follows that the dual space $(B^1)^*$ of $B^1$ must be a superset of $BMO$. A natural superset candidate  of $BMO$ is therefore the space $BMO^w$ since $BMO\subset BMO^w$. So essence, that $BMO^w$ is the dual of $B^1$ would  also prove that $B^1\subset \HH^1$ with a strict inclusion.  Moreover, it was already proved that $(B^1)^*\cong \Lambda_*'$, where $ \Lambda_*'$ is  the space of derivative (in the sense of distributions) of functions in the Zygmund class $\Lambda_*$, see for example \cite{DeSouza1980} and \cite{Zygmund2002}. Therefore, if $(B^1)^*\cong BMO^w$, then by transition we would have $\Lambda'_*\cong BMO^w$. \\
Henceforth, we will adopt the following notations: $\DD=\set{z\in \C: \abs{z}<1}$ is the open unit disk and let $\TT=\set{z\in \C: \abs{z}=1}$ is the unit sphere. For an integrable function $f$ on a measurable set $A$, and the Lebesgue measure $\lambda$ on $A$, we will write $\ds \aint_A f(\xi)d\lambda(\xi):=\frac{1}{\lambda(A)} 
\int_A f(\xi)d\lambda(\xi)$. We will start  by recalling the  necessary definitions and important results. The interested reader can see for example \cite{Gill2017} for more information.\\
%The Hardy spaces were introduced by Frigyes. Riesz in 1923 and named after Godfrey  Harold Hardy, and are defined as follows:
\begin{defn}
Let $0<p<\infty$ be a real number.  The Hardy Space $\HH^p:=\HH^p(\DD)$ is the space of holomorphic functions $f$ defined on $\DD$ and satisfying 
\[ \norm{f}_{\HH^p}:=\sup_{0\leq r<1}\left(\frac{1}{2\pi}\int_0^{2\pi} \abs{f(re^{i\xi})}^p d\lambda(\xi) \right)^{\frac{1}{p}}<\infty\;. \]
\end{defn}
  \noindent  Let $\ds f\in L^1_{\mbox{loc}}(\R^n),  Q$ be a hypercube in $\R^n$,  and $\lambda$  be the Lebesgue measure on $\R^n$ for some $n\in \N$.\\
 %We will consider \[\ds \aint_A:f=\frac{1}{\lambda_n(A)}\int_A f(y)d\lambda_n(y)\;.\]
Put \[ \ds f^{\#} _{Q}=\aint_Q f(\xi)d\lambda(\xi),\quad f_{Q}=\abs{\aint_Q f(\xi)d\lambda(\xi)} \;.\]
For $\ds f\in L^1_{\mbox{loc}}(\R^n)$ and $x\in \R^n$, we define
\begin{eqnarray}
\ds M^{\#} (f)(x)&=&\sup_{Q\ni x} \;\aint_Q \abs{f(\xi)-f^{\#}_Q}d\lambda(\xi)\\
 M(f)(x)&=&\sup_{Q\ni x}\; \abs{\aint_Q [f(\xi)-f_Q]d\lambda(\xi)} \\
 mf(x)&=& \sup_{Q\ni x}\; f_Q\;.
 \end{eqnarray}
where the supremum is taken over all hypercubes $Q$ containing $x$. 
Now we can define the space of functions of  bounded mean oscillation and its weak counterpart.
\begin{defn} The space of functions of bounded mean oscillation is defined as the space of locally integrable functions $f$ for which the operator $M^{\#}$ is bounded, that is, 
\[BMO(\R^n)=\set{f\in L^1_{\mbox{loc}}(\R^n): M^{\#}(f) \in L^{\infty}(\R^n)}\;.\]
We can endowed $BMO(\R^n)$ with the norm \[\norm{f}_{BMO(\R^n)}=\norm{M^{\#}(f)}_{\infty}=M^{\#}(f)(x)\;.\]

The space of functions of weak bounded mean oscillation is defined as the  space of locally integrable functions $f$ for which the operator $M$ is bounded, that is, 
\[BMO^w(\R^n)=\set{f\in L^1_{\mbox{loc}}(\R^n): Mf \in L^{\infty}(\R^n)}\;.\]
%We can endowed $BMO^w(\R^n)$ with the semi-norm \[\norm{f}_{BMO^w(\R^n)}=\norm{M^{\#}(f)}_{\infty}\;.\]
\end{defn}
\begin{Rem} It  follows from the above definitions that $BMO(\R^n)\subseteq BMO^w(\R^n)$. 
\end{Rem}
\noindent Let us recall the definition of the space of functions of vanishing mean oscillation $VMO(\R^n)$ and introduce the space of functions of weak vanishing mean oscillations $VMO^w(\R^n)$. 
\begin{defn} Let $f\in L^1_{loc}(\R^n)$. \\
\[ \mbox{ $f\in VMO(R^n)$ if }\quad  \lim_{\lambda(Q)\to 0}\; \aint_Q \abs{f(\xi)-f_Q}d\lambda(\xi)=0\;.\]
\[ \mbox{$f\in VMO^w(\R^n)$ if  } \quad \lim_{\lambda(Q)\to 0}\; \abs{\aint_Q[f(\xi)-f_Q]d\lambda(\xi)}=0\;.\]
%The space of such functions will be referred to as $VMO^w(\R^n)$.
\end{defn}
\begin{Rem} It follows from the above definition that $VMO(\R^n)$ is a subspace of  $VMO^w(\R^n)$ which is itself a subspace of $BMO^w(\R^n)$. We will show (see Theorem \ref{theo7} below) that $VMO^w$ is in fact closed subspace of $BMO^w$.
\end{Rem}
\noindent Henceforth, $BMO(\R^n), BMO^w(\R^n)$, and $VMO^w(\R^n)$ will simply be referred to as $BMO, BMO^w$, and $VMO^w$.

Now we consider $A(\DD)$ as the space of analytic functions defined on the unit disk $\DD$. Following the work of Girela in \cite{Girela1999} on the space of analytic functions of bounded means oscillations, we introduce their weak counterpart.  Before, we recall that the Poisson Kernel is defined as \[\ds P_r(\theta)=\mbox{Re}\left( \frac{1+re^{i\theta}}{1-re^{i\theta}}\right)\;.\]

\begin{defn}The space of  analytic functions of bounded mean oscillation is defined as 
\[\ds BMOA(\DD)=\set{F\in A(\DD); \exists f\in BMO(\TT): F(re^{i\theta})=\frac{1}{2\pi}\int _{\TT}P_r(\theta-\xi)f(e^{i\xi})d\lambda(\xi)}\;.\]
We can endowed $BMO(\DD)$ with the norm 
%\[\norm{F}_{BMOA(\DD)}:=\sup_{\abs{z}<1}\; \left(\frac{1}{2\pi}\int _{-\pi}^{\pi}\frac{1-\abs{z}^2}{\abs{z-e^{i\xi}}^2}\abs{f(e^{i\xi})-F(z)}d\xi\right)<\infty\;.\]
\[\norm{F}_{BMOA(\DD)}:=\abs{F(0)}+\sup_{\underset{\theta \in \TT}{0\leq r<1}}\; \left(\frac{1}{2\pi}\int _{\TT}P_r(\theta-\xi)\abs{f(e^{i\xi})-F(re^{i\theta})}d\lambda(\xi)\right)<\infty\;.\]
 %where $\ds  F(a)=\frac{1}{2\pi}\int _{-\pi}^{\pi}\frac{1-\abs{a}^2}{\abs{a-e^{i\theta}}^2}f(e^{i\theta})d\theta$.
 In other words, $BMOA(\DD)$ is the space of Poisson integrals of functions in $BMO(\TT)$.
 \end{defn}
 \noindent We can now define the space $BMOA^w$ of analytic function of weak bounded mean oscillation.
 \begin{defn}
 An analytic   function $F$ on $\DD$ is said to be of weak bounded mean oscillation if there exists $f\in BMO^w(\TT)$ such that 
 \[ F(re^{i\theta})=\frac{1}{2\pi}\int _{\TT}P_r(\theta-\xi)f(e^{i\xi})d\lambda(\xi)\;.\]
 We endow $BMOA^w(\DD)$ with the norm
\[\norm{F}_{BMOA^w(\DD)}:=\abs{F(0)}+\sup_{\underset{\theta \in \TT}{0\leq r<1}}\; \left(\frac{1}{2\pi}\abs{\int _{\TT}P_r(\theta-\xi) \sbb{f(e^{i\xi})-F(re^{i\theta})}d\lambda(\xi)}\right)<\infty\;.\]
 \end{defn}
\noindent  We recall the definition of  special atom space  $B^1$, see \cite{Kwessi2019}\;.

%The  weighted special atom (of type 1)  on $J=L\cup R$, a sub-interval  of $I^d$, is defined as \[\ds b_w(\bm{\xi})=\frac{1}{w(J)}\left\{ \chi_{_{R}}(\bm{\xi})-\chi_{_{L}}(\bm{\xi})\right\},\] where $\ds w(J)=\int_Jw(\bm{\xi})d\bm{\xi}$  and \textcolor{red}{$\ds R=\bigcup _{j=1}^{2^{d-1}}J_{i_j}$ for some $(i_1, i_2, \cdots, i_{2^{d-1}}\subset \set{1,1, \cdots, 2^d})$ with $i_1<i_2<\cdots <i_{2^{d-1}}$ and $L=J\backslash R$. $\set{J_1, J_2,\cdots, J_{2^d}}$ is the collection of sub-cubes of $J$, cut by the hyperplanes $x_1=a_1, x_2, \cdots, x_d=a_d$.} $\chi_A$ represents the characteristic function of set $A$. 
 \begin{defn}
For $n\geq 1$, we consider the hypercube of $\R^n$  given as $\ds J=\prod_{j=1}^n[a_j-h_j,a_j+h_j]$ where $a_j, h_j$ are real numbers with $h_j>0$. Let  $\phi \in L^1(J)$ with $\ds \phi(J)=\int_J \phi(\xi)d\lambda(\xi)$\;. \\ The  special atom (of {\bf type 1}) is a function $b:I\subseteq J\to \mathbb{R}$ such that
\begin{eqnarray*}
b(\xi)&=& 1 ~ \textrm{on $J\backslash I$ or }\\
b(\xi)&=&\dfrac{1}{\phi(I)}\sbb{\chi_{_R}(\xi)-\chi_{_L}(\xi)}\;,
\end{eqnarray*}
where  
\begin{itemize}
\item $\ds R=\bigcup _{j=1}^{2^{n-1}}I_{i_j}$ for some $i_1, i_2, \cdots, i_{2^{n-1}}\in  \set{1,2, \cdots, 2^n}$ with $i_1<i_2<\cdots <i_{2^{n-1}}$ and $L=I\backslash R$\;. 
\item $\set{I_1, I_2,\cdots, I_{2^n}}$ is the collection of sub-cubes of $I$, cut by the hyperplanes $x_1=a_1, x_2=a_2, \cdots, x_n=a_n$. 
\item $\chi_A$ represents the characteristic function of set $A$\;.
\end{itemize}
\end{defn}
\begin{defn}
\noindent The  special atom space is defined as 
 
 \[B^1=\set{f:J\to \mathbb{R}; f(\xi)=\sum_{n=0}^{\infty} \alpha_n b_n(\xi); \sum_{n=0}^{\infty} \abs{\alpha_n}<\infty}\;,\]
where the $b_n$'s are special atoms of  {\bf type 1}.  \\
\noindent $B^1$ is endowed with the norm 
$\ds \norm{f}_{B^1}=\inf \sum_{n=0}^{\infty}\abs{\alpha_n},\quad \mbox{where the infimum is taken over all representations of $f$}\;. $
\end{defn}

\begin{figure}[H]
\begin{center}
 \resizebox*{15cm}{!}{
\begin{tabular}{cc}
{\bf (a)} & {\bf (b)}\\
\centering\includegraphics[scale=0.8]{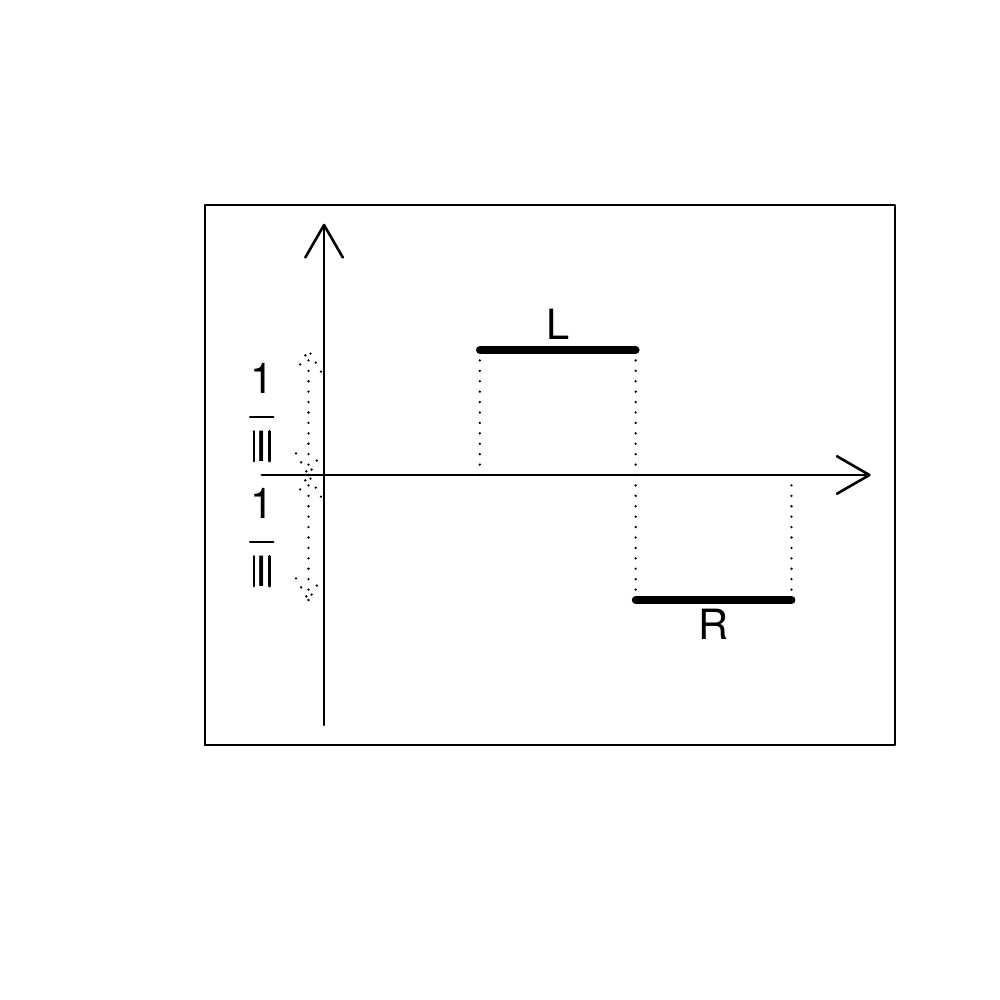} & \includegraphics[scale=0.45]{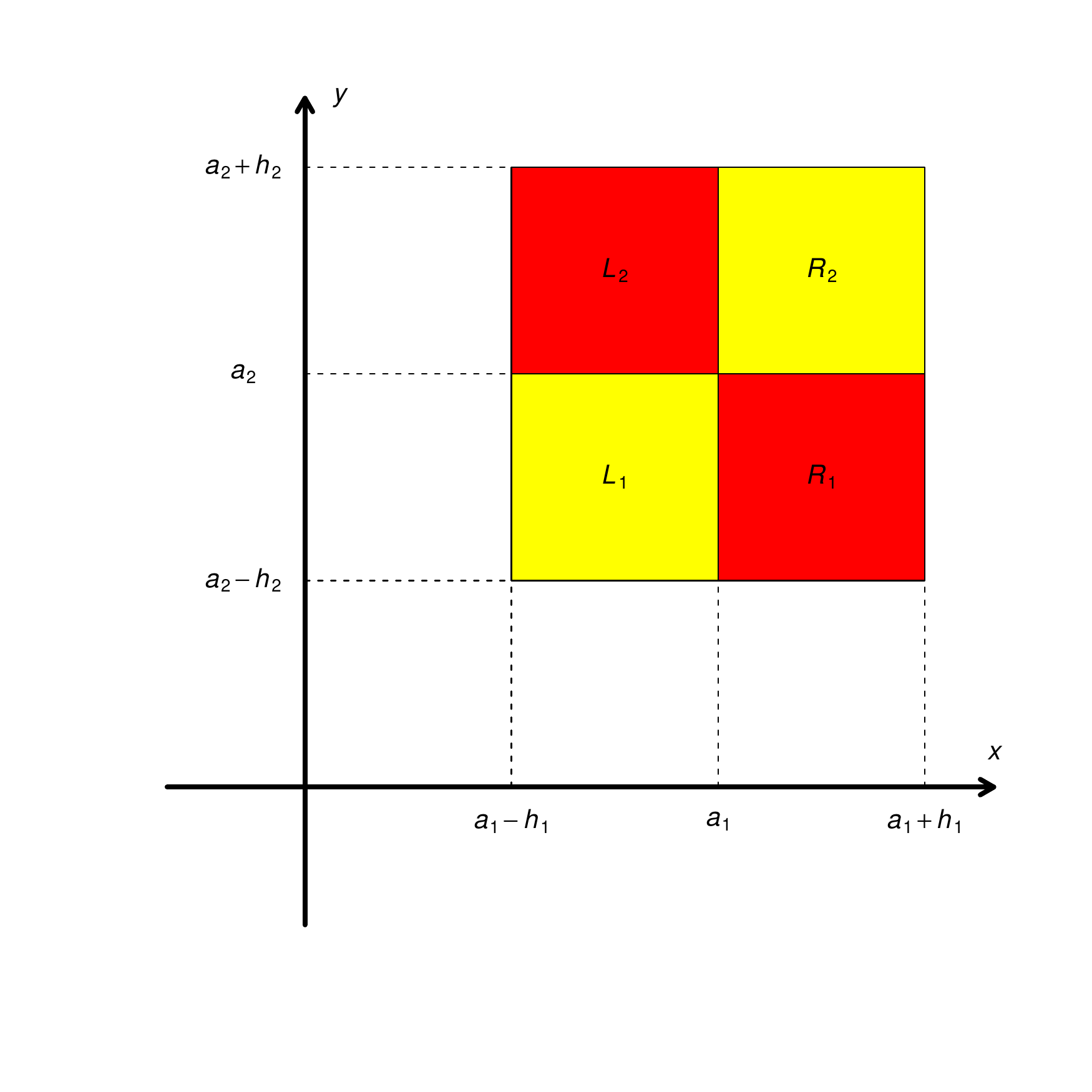}\\
{\bf (c)} & \\
 \includegraphics[scale=0.5]{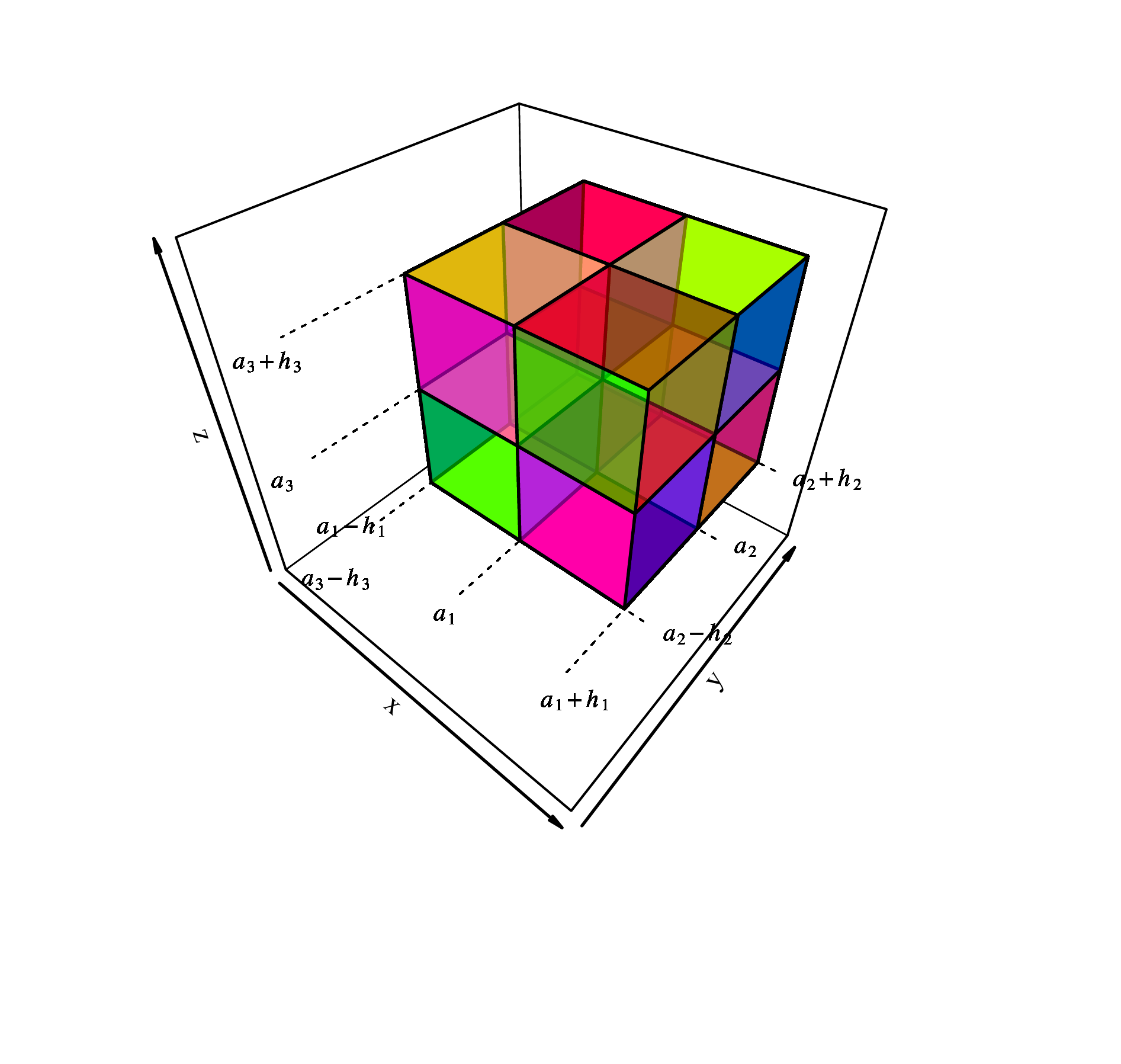}
 \end{tabular}
 }
 \end{center}
 \caption{Illustration of the the special atom, for $n=1$ in {\bf (a) }, $n=2$ in {\bf (b)}, and $n=3$ in {\bf (c)}.}
 \label{fig1}
\end{figure}

\noindent Now we define the Zygmund class of functions.
\begin{defn}  Let $k\in \N$. A function $f$ is said to be in the Zygmund class  $\Lambda_*^k(\R^n)$ of functions of order $k$ if  $f\in C^{k-1}(\R^n)$ and 
%\[\norm{f}_{\Lambda_*^k(\R^n)}=\norm{f}_{C^{k-1}(\R^n)}+\sum_{\abs{\alpha}=k} \sup_{\xi,h}\; \frac{\abs{\partial^{\alpha} f(\xi+h) + \partial^{\alpha} f(\xi-h)-2\partial^{\alpha} f(\xi)}}{\abs{h}}<\infty\;.\]
\[\norm{f}_{\Lambda_*^k(\R^n)}=\sum_{\abs{\alpha}=k} \sup_{x,h}\; \frac{\abs{\partial^{\alpha} f(x+h) + \partial^{\alpha} f(x-h)-2\partial^{\alpha} f(x)}}{\abs{h}}<\infty\;.\]
In particular, for $k=1$, we have $\Lambda_*:=\Lambda_*^1(\R^n)$,  and hence
\[ \Lambda_*=\set{f\in C^0(\R^n): \norm{f}_{\Lambda_*}:=\sup_{x,h>0}\; \frac{\abs{f(x+h)+f(x-h)-2f(x )}}{2h}<\infty}\;.  \]
\end{defn}
\noindent One important note about the space $\Lambda_*$ is that it contains the so-called Weierstrass functions that are known to be continuous everywhere but nowhere differentiable. Therefore, the space $\Lambda_*^k$ is the space of derivatives of functions in  $\Lambda_*^{k-1}$, where the derivative is taken in the sense of distributions. 
Another equivalent way to see $\Lambda_*^k$ is to consider functions of $\Lambda_*^{k-1}$ that are either differentiable or limits of convolutions with the Poisson kernel, that is, $\ds f(\xi)=\lim_{r\to 1}(f*P_r)(\xi)$ where $P_r(\theta)$ is the Poisson kernel. 
%%%%%%%%%%%%%%%%%%%%%%%%%%%%%%%%%%%%%%%%%%%%%%%%%%%%%%%%%%%%%%%%%%%%%%%%%%%%%%%%%%%%%%%%%%%
%%%%%%%%%%%%%%%%%%%%%%%%%%%%%%%%%%%%%%%%%%%%%%%%%%%%%%%%%%%%%%%%%%%%%%%%%%%%%%%%%%%%%%%%%%%
\section{Main Results}
%%%%%%%%%%%%%%%%%%%%%%%%%%%%%%%%%%%%%%%%%%%%%%%%%%%%%%%%%%%%%%%%%%%%%%%%%%%%%%%%%%%%%%%%%%%
%%%%%%%%%%%%%%%%%%%%%%%%%%%%%%%%%%%%%%%%%%%%%%%%%%%%%%%%%%%%%%%%%%%%%%%%%%%%%%%%%%%%%%%%%%%

%Let $x\in \R^n$ and let $Q\subseteq \R^n$ be a hypercube containing $x$. We define \[\norm{mf}_{\infty}=\sup_{Q\ni x} f_Q=\sup_{Q\ni x} \abs{\aint_Q f(\xi)d\lambda(\xi)},\] where the supremum is taken over all hypercube containing $x$. 

Our first result is about the constant $f_Q$ in the definition of $BMO^w$. If fact, the constant $f_Q$ can be replaced with any non-negative constant. The same can be said about  $BMO$ as well.
\begin{prop} \label{prop1} Let $f\in L^1_{loc}$\;.
\begin{enumerate}
\item[(1)] For any non-negative  real number $\alpha$, we have 
\begin{equation}\label{eqn2.1}
\sup_{Q\ni x}\; \abs{\aint_Q[f(\xi)-f_Q]d\lambda(\xi)} \leq2 \sup_{Q\ni x}\;\abs{\aint_Q [f(\xi)-\alpha]d\lambda(\xi)}\;.
\end{equation} 
\item[(2)] $f\in BMO^w$  if and only if  for any $x\in \R^n$ and  any cube $Q\ni x$, there exists  is a non-negative number $\alpha \in Q$ such that 
\begin{equation} \label{eqn2.2}
\sup_{Q\ni x}\; \abs{\aint_Q[f(\xi)-\alpha]d\lambda(\xi)}<\infty.
\end{equation}
\end{enumerate}
\end{prop}
\begin{proof}
(1) To prove assertion (1), fix $x\in \mathbb{R}^n$ and a cube  $Q\subseteq \R^n$ containing $x$. 
Observe that for every non-negative $\alpha$ we have,

%To prove the first part, suppose   $x\in \R^n$ and let  a cube

\[\abs{\aint_Q[f(\xi)-f_Q]d\lambda(\xi)}\leq 2 \abs{\aint_Q [f(\xi)-\alpha]d\lambda(\xi)}\;.\]
Indeed, for any non-negative  real number $\alpha$,  we have 
\begin{eqnarray*}
\abs{\aint_Q[f(\xi)-f_Q]d\lambda(\xi)}&=& \abs{\aint_Q [f(\xi)-\alpha]d\lambda(\xi)+\aint_Q [\alpha-f_Q]d\lambda(\xi)}\\
&\leq & \abs{\aint_Q [f(\xi)-\alpha]d\lambda(\xi)}+\abs{\aint_Q [\alpha-f_Q]d\lambda(\xi)}\\
&\leq & \abs{\aint_Q [f(\xi)-\alpha]d\lambda(\xi)}+\abs{\abs{\aint_Q \alpha d\lambda(\xi)}-\abs{\aint_Q f(\xi)d\lambda(\xi)}}\\
&\leq & \abs{\aint_Q [f(\xi)-\alpha]d\lambda(\xi)}+\abs{\aint_Q [\alpha- f(\xi)]d\lambda(\xi)}\\
&\leq &2\abs{\aint_Q [f(\xi)-\alpha]d\lambda(\xi)}\;.\\
\end{eqnarray*}
To conclude, we taking the supremum over of all cubes $Q\ni x$. \\

(2) Assertion (2) follows immediately from (1).
%Let $f\in BMO^w$,  then  \[ \sup_{Q\ni x}\; \abs{\aint_Q[f(\xi)-f_Q]d\lambda(\xi)}<\infty\;.\]
%Thus \eqref{eqn2.2} is satisfied with $\alpha=\inf\; f_Q$.\\
%``$\Longleftarrow$" 
%Suppose  there exists  is a non-negative number $\alpha$ such that 
%$\ds \sup_{Q\ni x} \;\abs{\aint_Q[f(\xi)-\alpha]d\lambda(\xi)}<\infty$, then from equation \eqref{eqn2.1} above,   $\ds \sup_{Q\ni x} \; \abs{\aint_Q[f(\xi)-f_Q]d\lambda(\xi)}\leq 2 \sup_{Q\ni x} \abs{\aint_Q[f(\xi)-\alpha]d\lambda(\xi)}<\infty$, that is, $f\in BMO^w$. \\
%On the other hand, if 
%We conclude that $\ds \sup_{Q\ni x} \; \abs{\aint_Q[f(\xi)-\alpha_Q]d\lambda(\xi)}<\infty$ with $\alpha_Q=f_Q
%\geq 0$\;.
\end{proof}

\noindent We  can define two equivalent  norms on $BMO^w$ and prove that endowed with these norms, $BMO^w$ is in fact a Banach space.
\begin{prop} Consider the following: For a every $f\in BMO^w$, we put \[\ds \norm{f}_{BMO^w}=\norm{mf}_{\infty}+\norm{Mf}_{\infty} \quad \mbox{and \quad $\ds \norm{f}^*_{BMO^w}=\norm{mf}_{\infty}+2\sup_{Q\ni x}\; \inf_{\alpha>0}\; \abs{ \aint_Q [f(\xi)-f_Q]d\lambda(\xi)}$\;.}\]
Then 
\[\norm{f}_{BMO^w}\leq \norm{f}^*_{BMO^w}\leq 2\norm{f}_{BMO^w}\;.\]
\end{prop}
%\noindent \textcolor{blue}{This is what I had before: \[  \norm{f}_{BMO^w}:\norm{f}_{L^1_{loc}}+\norm{Mf}_{\infty}\;.\] }
\begin{proof} The proof is an immediate consequence of Proposition \ref{prop1}. We note that the proof can also be obtained from the Closed-graph theorem, but that will require to first prove that endowed with the two norms, $BMO^w$ is a Banach space. 
\end{proof}
\begin{theom} The space $(BMO^w, \norm{~\cdot~}_{BMO^w})$ is complete.
%\begin{enumerate}
%\item[(1)] 
%\textcolor{blue}{Let $f\in BMO^w$ and define \[  \norm{f}_{BMO^w}:\norm{f}_{L^1_{loc}}+\norm{Mf}_{\infty}\;.\] }%\[\ds \norm{f}_{BMO^w}:=\norm{f}_{L^1_{loc}}+2\sup_{Q\ni x}\; \inf_{\alpha_Q>0} \abs{\aint_Q [f(\xi)-\alpha_Q]d\lambda(\xi)}\;.\]  %$\norm{f}_{BMO^w}:=\norm{f}_{L^1_{loc}}+\norm{Mf}_{\infty}$. 
%Then 
%$\norm{~\cdot~}_{BMO^w}$ is a norm on $BMO^w$.
%\item[(2)] 

%\end{enumerate}

\end{theom}
%\noindent Before we prove the theorem, we observe the following:
%\begin{Rem} From proposition \ref{prop1}, it follows that
%\[ K=\bigcap_{Q\subseteq \R^n}\; \bigcup _{\alpha\in Q}\; \set{ \alpha>0:  \sup_{Q\ni x} \abs{\aint_Q [f(\xi)-\alpha] d\lambda(\xi)}}\neq \emptyset\;.\]
%If $\ds \inf\; K<\infty$, we  can also define a norm on $BMO^w$ as     \[ \ds \norm{f}^1_{BMO^w}:=\norm{f}_{L^1_{loc}}+2\sup_{Q\ni x}\; \inf_{\underset{\alpha>0}{\alpha \in Q}} \;\abs{\aint_Q [f(\xi)-\alpha]d\lambda(\xi)}\;.\] However, the condition that $\inf\; K<
%\infty$ appears to be very strong and my not be very realistic. Moreover, by Proposition \ref{prop1} the norm $\norm{~\cdot~}^1_{BMO^w}$ obtained will be much bigger than $\norm{~\cdot~}_{BMO}$ defined above. 
%\end{Rem}
\begin{proof}
(1). In the proof that $\norm{~\cdot~}_{BMO^w}$ is a norm, homogeneity and the triangle inequality are easy to prove. As for positivity, we note that $\ds \norm{f}_{BMO^w}=0\Longleftrightarrow \sup_{Q\ni x}\;f_Q=0$  and $f(\xi)=f_Q$ on all cubes $Q\ni x$. It follows immediately that   $f= 0 $\;.\\

\noindent (2). Let $\set{f_n}_{n\in \N}$ be a Cauchy sequence in $BMO^w$. Let $\epsilon>0$ and $N\in \N$ such that $\forall n,m \in N$, we have $\norm{f_n-f_m}_{BMO^w}<\epsilon$.  That is, 
\begin{equation}\label{eqnn}
\ds \sup _{Q\ni x}\sbb{(f_n-f_m)_Q+\abs{\aint_Q [(f_n-f_m)(\xi)-(f_n-f_m)_Q]d\lambda(\xi)}}<
\epsilon\;.
\end{equation}
In particular,  from  \eqref{eqnn},  we have that $\ds \sup _{Q\ni x}\;(f_n-f_m)_Q<\epsilon$, therefore
\begin{eqnarray*}
\abs{f_{n,Q}-f_{m,Q}}&=& \abs{\abs{\aint_Q f_n(\xi)d\lambda(\xi)}-\abs{\aint_Q f_m(\xi)d\lambda(\xi)}}\\
&\leq & \abs{\aint_Q [f_n-f_m](\xi)d\lambda(\xi)}=(f_n-f_m)_Q\\
&\leq & \sup_{Q\ni x}\;(f_n-f_m)_Q<\epsilon\;.
\end{eqnarray*}
Hence for fixed $Q$,  $\set{f_{n,Q}}$ is Cauchy sequence in $\R$. Let $\ds f_Q=\lim_{n\to \infty}f_{n,Q}$. 
We note from the above inequalities that 
\[(f_n-f_m)_Q=\abs{\aint_Q [f_n-f_m](\xi)d\lambda(\xi)}\geq \abs{f_{n,Q}-f_{m,Q}}\geq f_{n,Q}-f_{m,Q}\;.\]
Therefore, given a cube $Q \subset \mathbb{R}^n$ containing  $x$, we have 
\begin{eqnarray*}
\abs{Mf_n(x)-Mf_m(x)}&\leq & \sup_{Q\ni x}\; \abs{\aint_Q [(f_n(\xi)- f_m(\xi))-(f_{n,Q}-f_{m,Q})]d\lambda(\xi)}\\
&\leq & \sup_{Q\ni x}\;  \abs{\aint_Q [(f_n(\xi)- f_m(\xi))-(f_n-f_m)_Q]d\lambda(\xi)}+\ \sup_{Q\ni x}\; \abs{(f_n-f_m)_Q-(f_{n,Q}-f_{m,Q})}\\
&\leq & \norm{f_n-f_m}_{BMO^w}+ \sup_{Q\ni x}\; \abs{(f_{n,Q}-f_{m,Q})}<2\epsilon\;.
\end{eqnarray*}
It follows that $\set{Mf_n}_{n\in \N}$ is a Cauchy-sequence in $L^{\infty}_{loc}$. Let $\ds h=\lim_{n\to \infty}Mf_n$.\\
\noindent Since $L^{\infty}_{loc}\subset L^1_{loc}$, we have $h\in L^1_{loc}$\;. 
\begin{eqnarray*}
h(x)&=&\lim_{n\to \infty} Mf_n(x)\\
&=& \lim_{n\to \infty} \; \sup_{Q\ni x} \abs{\aint_Q [f_n(\xi)-f_{n,Q}]d\lambda(\xi)}\\
&=& \sup_{Q\ni x} \abs{\aint_Q [\; \lim_{n\to \infty} f_n(\xi)-f_{Q}]d\lambda(\xi)}\;.
\end{eqnarray*}
Since $h(x)$ is finite on any cube $Q\ni x$, that follows that  $\ds f(\xi)=\lim_{n\to \infty} f_n(\xi)$ is finite a.e.  on $Q$. Thus $h=Mf$, for some $f\in BMO^w$ and  $\ds \lim_{n\to \infty} \norm{f_n-f}_{BMO^w}\to 0$.
\end{proof}

\begin{theom}[H\"{o}lder's type inequality]
Let $g\in BMO^w$ and a hyper-cube $J\subset \R^n$. Consider the following operator   $T_g: B^1\to \R$ given by $\ds T_g(f)=\int_J f(\xi)g(\xi)d\lambda(\xi)$. Then, $T_g \in (B^1)^*$ with $\norm{T_g}_{(B^1)^*}\leq \norm{g}_{BMO^w}$\;.\\
 Moreover, the operator $H: BMO^w\to (B^1)^*$  defined as $H(g)=T_g$ is onto. 
%that is, $\forall G \in (B^1)^*$, there exists $g\in BMO^w$ such that $T(g)=T_g=G$\;.

\end{theom}
\begin{proof}
 By linearity of the integral, $T_g$ is  a linear.  To start, we study the action of this operator on special atoms. Indeed, let $\xi\in I\subseteq J$ and  suppose $f(\xi)=b(\xi)=\frac{1}{\phi(I)}[\chi_{_R}(\xi)-\chi_{_L}(\xi)]$, where $R,L$ are sub-cubes of $I$ such that $I=R\cup L$ and $R\cap L=\emptyset$. Therefore we have:
\begin{eqnarray*}
T_g(b)&=&\int_J b(\xi)g(\xi)d\lambda(\xi)= \int_I b(\xi)g(\xi)d\lambda(\xi)\\
%&=&\frac{1}{\lambda(J)} \int_I [\chi_{_R}(\xi)-\chi_{_L}(\xi)] g(\xi)d\lambda(\xi)\\
%&=& \frac{1}{\lambda(J)} \int_I [\chi_{_R}(\xi)-\chi_{_L}(\xi)] (g(\xi)-g_J)d\lambda(\xi), \quad \mbox{since $\ds \int_I b(\xi)g_Jd\lambda(\xi)=0$}\\
&=&  \int_I b(\xi)[g(\xi)-g_I]d\lambda(\xi), \quad \mbox{since $\ds \int_I b(\xi)g_Id\lambda(\xi)=0$}\;.\\
\end{eqnarray*}
Taking the supremum over all $I\ni x$ and using the fact that $b(\xi)\leq \frac{1}{\phi(I)}$, we have 
\begin{equation} \label{eqn1}
\abs{T_g(b)}\leq \sup_{\underset{I\subseteq J}{I\ni x}} \inf_{\alpha>0}\abs{\aint_I[g(\xi)-\alpha]d\lambda(\xi)}\leq \norm{g}_{BMO^w}\;. 
\end{equation}
Now suppose $\ds f(\xi)=\sum_{n=1}^{\infty}c_nb_n(\xi)$ with $\ds \sum_{n=1}^{\infty} \abs{c_n}<\infty$ is an element of $B^1$, where the $b_n$'s  are special atoms defined on sub-cubes $I_n$ of $J$ with $I_n=R_n\cup L_n$ and $R_n\cap L_n=\emptyset$. Let $\ds I=\bigcup_{n=1}^{\infty} I_n$.  We have for $\alpha>0$
\begin{eqnarray*}
T_g(f)&=&\int_J \rb{\sum_{n=1}^{\infty} c_nb_n(\xi)}g(\xi)d\lambda(\xi)\\
&=& \int_I \rb{\sum_{n=1}^{\infty} c_nb_n(\xi)}\sbb{g(\xi)-\alpha}d\lambda(\xi)\\
&=&\sum_{n=1}^{\infty} c_n \int_I b_n(\xi)\sbb{g(\xi)-\alpha}d\lambda(\xi)\\
&=&\sum_{n=1}^{\infty} c_n \int_{I_n} b_n(\xi)\sbb{g(\xi)-\alpha}d\lambda(\xi)\;. %=\rb{\sum_{n=1}^{\infty} c_n}\cdot T_g(b_n)\\
\end{eqnarray*}
It follows from  \eqref{eqn1} that 
\[\abs{T_g(f)}\leq \rb{\sum_{n=1}^{\infty} \abs{c_n}} \cdot \norm{g}_{BMO^w}\;. \]
Talking the infimum over all representations of $f$ yields:
\[ \abs{T_g(f)}\leq \norm{f}_{B^1}\cdot \norm{g}_{BMO^w}\;.\]
Therefore $T_g$ is a bounded linear operator on $B^1$ with 
\begin{equation}\label{eqn11}
\norm{T_g}_{(B^1)^*}=\sup_{\norm{f}_{B^1}=1}\abs{T_g(f)}\leq \norm{g}_{BMO^w}\;.
\end{equation}
Now suppose that $T$ is a bounded linear functional on $B^1$, that is, $T\in (B^1)^*$. We want to show that there exists a function $g\in BMO^w$ such that $\ds T(f)=T_g(f)=\int_J f(\xi)g(\xi)d \lambda(\xi)$. That $T\in (B^1)^*$  implies the existence of an absolute constant $C$ such that 
\begin{equation}\label{eqn2}
\abs{T(f)}\leq C\norm{f}_{B^1}, \quad  \forall f\in B^1.
\end{equation}
Recall that a function $G:J\to \R$ is said to be absolutely continuous on $J$ if for every positive number $\epsilon$, there exists a positive number $\delta$, such that for a finite sequence of pairwise disjoint sub-cubes $I_n=(x_n, y_n)$ of $J$, \begin{equation}\label{eqn3}
\ds \sum_{n=1}^{\infty} (y_n-x_n)<\delta\implies \sum_{n=1}^{\infty} \abs{G(y_n)-G(x_n)}<\epsilon\;.\end{equation}
Suppose such a sequence of sub-cubes exists. Now consider $G(x)=T\rb{\chi_{_{[x-h, x+h)}}}$ for some real number $h>0$. Then, given an cube $I_n$ and a real number $h_n$, we have by linearity of $T$ that 
\[G(y_n)-G(x_n)=T\rb{\left[\chi_{_{[y_n-h_n, y_n+h_n)}}-\chi_{_{[x_n-h_n, x_n+h_n)}}\right]}\;.\]
Now if we define $L_n=[x_n-h_n,x_n+h_n)$ and $R_n=[y_n-h_n, y_n+h_n)$, we see that  $R_n\cap L_n=\emptyset$ their union forms a single cube $J_n$ if we choose $h_n=(y_n-x_n)/2$.  Moreover, in that case, the length of the cube $J_n$ is $y_n+h_n-x_n+h_n=2(y_n-x_n)$.  Therefore by linearity
\[G(y_n)-G(x_n)=2(y_n-x_n)T\rb{\frac{1}{2(y_n-x_n)}\left[\chi_{_{[y_n-h_n, y_n+h_n)}}-\chi_{_{[x_n-h_n, x_n+h_n)}}\right]}\;.\]
We observe that $\ds b_n(\xi)=\frac{1}{2(y_n-x_n)}\left[\chi_{_{[y_n-h_n, y_n+h_n)}}(\xi)-\chi_{_{[x_n-h_n, x_n+h_n)}}(\xi)\right]$ is a special atom with norm in $B^1$ equal to 1. It follows that using \eqref{eqn2}
\[\abs{G(y_n)-G(x_n)}\leq 2(y_n-x_n)\abs{T(b_n)}\leq 2C(y_n-x_n)\norm{b_n}_{B^1}=2C(y_n-x_n)\;.\]
Hence, 
\[\sum_{n=1}^{\infty} \abs{G(y_n)-G(x_n)}\leq 2C\sum_{n=1}^{\infty}(y_n-x_n)\;.\]
We conclude by noting that given $\epsilon>0$, \eqref{eqn3} is satisfied if we choose $\ds \delta=\frac{\epsilon}{2C}$. We conclude that the function $G$ is absolutely continuous on $J$. Therefore, $G$ is differentiable almost everywhere, that is, there exists $g\in L^1$ such that $\ds G(x)=\int_a^x g(\xi) d\lambda(\xi)$ for all cubes $I=[a,b]\subseteq J$.    Let $I\ni x$ be a sub-cube of $J$. Therefore, $\ds \sup_{I\ni x}\frac{1}{\lambda(I)} \abs{\int_I g(\xi) d\lambda(\xi)}<\infty$ because  an absolutely continuous function is a function with bounded variation. Since $g\in L^1$, we have that $\ds g_I=\frac{1}{\lambda(I)}\abs{\int_I g(\xi) d\lambda(\xi)}<\infty. $ It follows that 

\[ Mg(x)=\sup_{I\ni x}\frac{1}{\lambda(I)} \abs{\int_I (g(\xi)-g_I) d\lambda(\xi)}\leq \sup_{I\ni x}\frac{1}{\lambda(I)} \abs{\int_I g(\xi)d \lambda(\xi)}+\sup_{I\ni x}\;g_I <\infty\;.\]
This proves that $g\in BMO^w$. 
%We have therefore shown that  for every bounded linear functional $T$ on $B^1$, there exists $g\in BMO^w$ such that $\ds T(f)=T_g(f)=\int_J f(\xi)g(\xi)d\lambda(\xi)\;. $ 
That is, $H: BMO^w\to (B^1)^*$ is onto with $H(g)=T=T_g$. Identifying  $T$ with the $g$, it follows  from \eqref{eqn11} that 
\begin{equation}\label{eqn4}
\norm{g}_{(B^1)^*}=\norm{T}_{(B^1)^*}=\norm{T_g}_{(B^1)^*}\leq \norm{g}_{BMO^w}\;.
\end{equation}

\end{proof}
\begin{Rem}
We observe that the above result can be obtained differently. In fact, we recall that it was proved  in \cite{DeSouza1980} that the dual space of $B^1$ is equivalent to  $\Lambda_*^2$. Let  $x\in J, h>0, I=[x-h,x+h]\subseteq J, L=[x-h,x)$,  and $R=[x,x+h]$. Let $\ds b(\xi)=\frac{1}{2h}[\chi_{_R}(\xi)-\chi_{_L}(\xi)]$. We have 
\begin{eqnarray*} \norm{f}_{(B^1)^*}\cong \norm{f}_{\Lambda_*^2}&=& \sup_{\underset{h>0}{I\ni x}}\; \abs{ \int_I b(\xi)\partial f(\xi) d\lambda(\xi)}\\
&=& \sup_{\underset{h>0}{I\ni x}}\; \abs{ \int_I b(\xi)(\partial f(\xi)-\partial f_I) d\lambda(\xi)}\\
%&\leq &\sup_{\underset{h>0}{I\ni x}} \abs{ \frac{1}{2h}\int_I (\partial f(\xi) -\partial f_I)d\lambda(\xi)},\quad \mbox{since $\ds b(\xi)\leq \frac{1}{2h}$}\\
&\leq &\norm{ f}_{BMO^w}\;.
\end{eqnarray*}
This shows that $BMO^w\subseteq \Lambda_*^2\cong (B^1)^*$.
\end{Rem}

\begin{theom} \label{mainTheo}
The dual space $(B^1)^*$  of $B^1$ is $BMO^w$ with $\norm{g}_{BMO^w}\cong\norm{g}_{(B^1)^*}$\;. 
\end{theom}
\noindent To prove Theorem \ref{mainTheo}, we recall that $B^1(\TT)$  has an analytic extension $B^1_A(\DD)$ to the unit disk. In fact, in  \cite{Desouza1989}, it was  shown that  $B^1_A(\DD)$ consists of functions $F$ that are Poisson integrals of functions in $B^1(\TT)$, that is, 
$\ds F(z)=\frac{1}{2\pi}\int_{-\pi}^{\pi}\frac{1+e^{-i\xi}z}{1-e^{-i\xi}z}f(e^{i\xi})d\lambda(\xi)$ where $f\in B^1(\TT)$. Moreover, the norm $\ds \norm{F}_{B^1_A(\DD)}=\int_0^1\int_{-\pi}^{\pi}\abs{F'(z)}dz$ is equivalent to the norm $\norm{f}_{B^1(\TT)}$.  This allows to identify $B^1_A(\DD)$ with $B^1(\TT)$ and thus the following:
\begin{prop} \label{Pop2} $B^1_A(\DD)$ can be identified with a closed subspace of $L^1(\TT)$. 
\end{prop}
\begin{proof} Let $\set{f_n}_{n\in \N}$ be a sequence in $B^1(\TT)$ that converges to $f$ in $L^1(\TT)$. We need to show that the Poisson integral of $f$ is in  $B^1_A(\DD)$. Since $f_n\in B^1(\TT), \forall\; n\in \N$, then $\ds F_n(z)=\frac{1}{2\pi}\int_{-\pi}^{\pi}\frac{1+e^{-i\xi}z}{1-e^{-i\xi}z}f_n(e^{i\xi})d\lambda(\xi)$ belongs to $B^1_A(\DD)$.  Let $F(z)$ be the Poisson integral of $f$.   We note that if  $ z=e^{i\theta}$, then $(\xi-\theta)^2-\frac{(\xi-\theta)^4}{2}\leq \abs{1-e^{-i\xi}z}^2\leq (\xi-\theta)^2$. 
Therefore, we have that $\ds F'(z)=\frac{1}{2\pi}\int_{-\pi}^{\pi}\frac{2e^{-i\xi}}{(1-e^{-i\xi}z)^2}f(e^{i\xi})d\lambda(\xi)$ and $\ds \abs{F'(z)}\leq \frac{1}{2\pi}\int_{-\pi}^{\pi}\frac{2}{\abs{1-e^{-i\xi}z}^2}\abs{f(e^{i\xi})}d\lambda(\xi)\leq C\norm{f}_{L^1(\TT)}$. 
It follows that , \[\ds \norm{F_n-F}_{B^1_A}=\int_0^1\int_{-\pi}^{\pi} \abs{F'_n(z)-F'(z)}dz\leq C\norm{f_n-f}_{L^1(\TT)}.\] Since $f_n\to f$ in $L^1(\TT)$, it follows that $F\in B^1_A$ and $F_n\to F$ in $B^1_A(\DD)$. The result follows by identifying $B^1_A(\DD)$ and $B^1(\TT)$. 
\end{proof}
%\begin{prop} Let $G\in BMOA^w(\DD)$ and $g\in BMO^w(\TT)$ such that $\ds G(re^{i\theta})=\frac{1}{2\pi}\int _{\TT}P_r(\theta-\xi)g(e^{i\xi})d\xi\;. $
% \[\norm{G}_{BMOA^w}\cong \norm{g}_{BMO^w}\;.\]
%\end{prop}
\noindent We note  that there is an extension of this result to the polydisk $\DD^n$ and polysphere $\TT^n$, see \cite{Kwessi2020}\;. 
\begin{proof}[Proof of theorem \ref{mainTheo} ]
It sufficient  to show that there exists a constant $M>0$ such that $\norm{T}_{(B^1(\TT))^*}\geq M \norm{g}_{BMO^w(\TT)}. $
The extension  to $\R^n$ is natural, using the results in \cite{Kwessi2020, Kwessi2019}\;.
With Proposition \ref{Pop2}, the proof follows along the lines of the proof of Proposition 7.3 in \cite{Girela1999}.\\
Let $T\in (B^1(\TT))^*$. Since $B^1(\TT)$ is a closed subspace of $L^1(\TT)$ by Proposition \ref{Pop2}, then by Hahn-Banach Theorem, $T$ can be extended to a bounded linear operator   $T'\in (L^1(\TT))^*$ with $\norm{T}_{(B^1(\TT))^*}=\norm{T'}_{(L^1(\TT))^*}$. Since $(L^1(\TT))^*=L^{\infty}(\TT)$, then there exists $g_0\in L^{\infty}(\TT)$ such that $\norm{g}_{L^{\infty}(\TT)}=\norm{T'}_{(L^1)^*}=\norm{T}_{(B^1)^*}$ and $\ds T(f)=\frac{1}{2\pi}\int_{\TT} \overline{g_0(e^{i\xi})}f(e^{i\xi})d\lambda(\xi)$, for all $f\in B^1_A(\DD)$. Note that here, we identify $f$ with its correspondent in $B^1_A(\DD)$.
Now let $\ds \sum_{n\in \N}A_ne^{in\xi}$ be the Fourier series of $g_0$. Since $g_0\in L^{\infty}(\TT)\subset L^2(\TT)$, we have that $\ds \sum_{n\in \Z}\abs{A_n}^2<\infty$\;. This means that $g_0$ is holomorphic. Let 
\[ g(z)=\frac{1}{2\pi i}\int_{-\pi}^{\pi}\frac{g_0(e^{i\xi})}{e^{i\xi}-z}d(e^{i\xi})= \frac{1}{2\pi}\int_{-\pi}^{\pi}\frac{g_0(e^{i\xi})}{1-e^{-i\xi}z}d \xi\;. \]
Since $\ds \frac{1}{1-e^{i\xi}z}=\sum_{n\in \N} e^{-i n \xi}z^n$ and $\ds A_n=\frac{1}{2\pi}\int_{-\pi}^{\pi} e^{-in\xi}g_0(e^{i\xi})d\lambda(\xi)$, we have 
\[ g(z)=\frac{1}{2\pi}\int_{-\pi}^{\pi}\frac{g_0(e^{i\xi})}{1-e^{-i\xi}z}d \xi=\sum_{n\in \N}A_nz^n \;.\]
This implies that $g\in H^2(\DD)$. For $\theta \in \R$,
$\ds \overline{g(e^{i\theta})}=\frac{1}{2\pi}\int_{-\pi}^{\pi} \frac{\overline{g_0(e^{i\xi})}}{1-e^{-i\xi}e^{i\theta}}d\lambda(\xi)$.  Moreover, given $f\in B^1_A(\DD) \subseteq H^1(\DD) $, and using the Cauchy integral formula, we have:
\begin{eqnarray*}
\frac{1}{2\pi}\int_{-\pi}^{\pi} \overline{g(e^{i\theta})} f(e^{i\theta})d\theta&=& \frac{1}{2\pi}\int_{-\pi}^{\pi} \overline{g_0(e^{i\theta})} \left( \frac{1}{2\pi}\int_{-\pi}^{\pi} \frac{f(e^{i\xi})}{1-e^{-i\xi}e^{i\theta}} d\lambda(\xi)\right)d\theta\\
&=& \frac{1}{2\pi}\int_{-\pi}^{\pi} \overline{g_0(e^{i\theta})} f(e^{i\theta})d\theta\\
&=& T(f)\;.
\end{eqnarray*}

 %$\ds T(f)=\frac{1}{2\pi}\int_{-\pi}^{\pi}f(e^{i\xi})\overline{g(e^{i\xi})}d \xi$, for all $f\in H^2(\DD)\subseteq B^1_A(\DD)$.  
 
 \noindent We also observe that 
\begin{equation} \label{eqn:g}\ds g(z)=\frac{1}{2\pi}\int_{-\pi}^{\pi}\frac{1}{2}\sbb{\frac{1+e^{-i\xi}z }{1-e^{-i\xi}z}+1} g_0(e^{i\xi})d \xi=\frac{1}{2\pi}\int_{-\pi}^{\pi}\frac{1+e^{-i\xi}z }{1-e^{-i\xi}z}G_0(e^{i\xi})d\lambda(\xi),
\end{equation} where 
$\ds G_0(e^{i\xi})=\frac{1}{2}\left( g_0(e^{i\xi})+\frac{1}{2\pi}\int_{-\pi}^{\pi} g_0(e^{i\xi})d\lambda(\xi) \right)=\frac{1}{2}\left( g_0(e^{i\xi})+A_0\right)\;.$\\
Put  $G_1=\mbox{Re}(G_0), G_2=\mbox{Im}(G_0)$, and 
\[ U(z)=\frac{1}{2\pi}\int_{-\pi}^{\pi}\frac{1+e^{-i\xi}z }{1-e^{-i\xi}z}G_1(e^{i\xi})d\lambda(\xi); V(z)=\frac{1}{2\pi}\int_{-\pi}^{\pi}\frac{1+e^{-i\xi}z }{1-e^{-i\xi}z}G_2(e^{i\xi})d\lambda(\xi)\;.\]
Then $g=U+iV$. Moreover,  $U$ and $V$ are analytic in $\DD$ since they  represent the Poisson integral of $G_1, G_2 \in L^{\infty}(\TT)\subseteq B^1(\TT)$. Observe  that $BMO\subseteq BMO^w$. Moreover, Theorem  3.2 in \cite{Girela1999} shows that $\norm{g}_{BMO}\cong \norm{g}_{BMOA}$. It therefore follows that  there exists a constant $C>0$ such that \[C\norm{g}_{BMO^w}\leq C \norm{g}_{BMO}\leq \norm{g}_{BMOA}\leq \norm{T'}_{(L^1)^*}=\norm{T}_{(B^1)^*}\;.\] 
\end{proof}

\subsection{Discussion}
We note that $BMO^w\subseteq \Lambda'_*$ with $\norm{f}_{\Lambda'_*}\leq C\norm{f}_{BMO^w}$ where $\norm{g}_{\Lambda'_*}=\norm{f}_{\Lambda_*}$ with $g'=f$ in the sense of distributions. Since $(B^1)^*\cong \Lambda'_*$, and from Theorem  \ref{mainTheo} above  $(B^1)^*\cong BMO^w$, it follows that $BMO^w\cong \Lambda'_*\;.$
The consequence is that there exists $c>0$ such that $c\norm{f}_{BMO^w}\leq \norm{f}_{\Lambda'_*}$, that is, those two norms are equivalent.  We finish by noting  $\Lambda'_*$  has an analytic characterization, so we would expect the analytic characterization of $BMO^w$ to also be equivalent to that of $\Lambda'_*$. \\
Another takeaway from Theorem \ref{mainTheo} is that it provides another proof that $BMO$ is strictly contained in $BMO^w$ otherwise we would have had $B^1\cong \HH^1$, which is not true. In other words, there exists $f\in \HH^1$ such that $f\notin B^1$. 
%%%%%%%%%%%%%%%%%%%%%%%%%%%%%%%%%%%%%%%%%%%%%%%%%%%%%%%%%%%%%%%%
%%%%%%%%%%%%%%%%%%%%%%%%%%%%%%%%%%%%%%%%%%%%%%%%%%%%%%%%%%%%%%%%
%%%%%%%%%%%%%%%%%%%%%%%%%%%%%%%%%%%%%%%%%%%%%%%%%%%%%%%%%%%%%%%%

\noindent Our next result is about the closeness of $VMO^w$ in $BMO^w$.
\begin{theom}\label{theo7} $VMO^w$ is a closed subspace of $BMO^w$. 
\end{theom}
\begin{proof} Let $\set{f_n}_{n\in \N}$ be a sequence in $VMO^w$ that converges to $f\in BMO^w$. Let us prove that $f\in VMO^w$. That $f_n\to f$ as $ n\to \infty$ in $BMO^w$ is equivalent to  $\ds \lim_{n\to \infty}\norm{f_n-f}_{BMO^w}=0$. The latter is also equivalent, by definition of the norm in $BMO^w$ to \[ \mbox{$\ds \lim_{n\to \infty}\; \sup_{Q\ni x}\;(f_n-f)_Q=0$ and} \quad  \ds \lim_{n\to \infty}M(f_n-f)(x)=0\;.\] Since $f\in BMO^w$, then $\ds \sup_{Q\ni x}\;\abs{\aint_Q [f(\xi)-f_Q]d\lambda(\xi)}<\infty$. Therefore
\begin{eqnarray*}
\abs{\aint_Q [f(\xi)-f_Q]d\lambda(\xi)}&=&\abs{\aint_Q [f(\xi)-f_n(\xi)+f_n(\xi)-f_{n,Q}+f_{n,Q}-f_Q]d\lambda(\xi)}\\
&\leq & \abs{\aint_Q [f(\xi)-f_n(\xi)]d\lambda(\xi)}+\abs{\aint_Q [f_n(\xi)-f_{n,Q}]d\lambda(\xi)}+\abs{\aint_Q [f_{n,Q}-f_Q]d\lambda(\xi)}\\
&\leq & \abs{\aint_Q [f(\xi)-f_n(\xi)]d\lambda(\xi)}+\abs{\aint_Q [f_n(\xi)-f_{n,Q}]d\lambda(\xi)}+\abs{f_{n,Q}-f_Q}\\
&\leq &  \abs{\aint_Q [f(\xi)-f_n(\xi)]d\lambda(\xi)}+\norm{f_n-f}_{BMO^w}\\
&\leq & \norm{f_n-f}_{L^1_{loc}}+\norm{f_n-f}_{BMO^w}\;.\\
%&\leq & 2\norm{f_n-f}_{L^1_{loc}}+\abs{\aint_Q [f_n(\xi)-f_{n,Q}]d\lambda(\xi)}\;.
\end{eqnarray*}

\noindent Since $\set{f_n}_{n\in \N}$ is a sequence of function is $BMO^w$, we have that $\ds \lim_{n\to \infty}\;\norm{f_n-f}_{BMO^w} =0$. Since $BMO^w\subset L^1_{loc}$, we have that $\ds \lim_{n\to \infty}\norm{f_n-f}_{L^1_{loc}}=0$. 
Hence  $VMO^w$ is a closed subspace of $BMO^w$.
%from the last inequality above, $\ds \lim_{\lambda(Q)\to 0}\abs{\aint_Q [f(\xi)-f_Q]d\lambda(\xi)}=0$, which shows that 
\end{proof} 
\begin{Rem} We note that it was proved in \cite{Girela1999} that $VMO$ is a closed subspace of $BMO$. A stronger result even showed that if we restrict ourselves to $\TT$, then the space of complex-valued and continuous functions $C(\TT)\subseteq VMO(\TT)$ and the $BMO(\TT)$-closure of $C(\TT)$ is precisely $VMO(\TT)$. It turns out this result is also true for $VMO^w$.

\end{Rem}
\begin{theom} The $BMO^w$-closure of $C(\TT)$ is precisely $VMO^w(\TT)$, that is 
\[ \overline{C(\TT)}^{BMO^w}=VMO^w(\TT) \;.\]
\end{theom}
\begin{proof} We observe that $C(\TT)\subseteq VMO(\TT)\subseteq VMO^w(\TT)\subseteq BMO^w(\TT)$. Therefore since $VMO^w(\TT)$ is closed in $BMO^w(\TT)$, we have that \[ \overline{C(\TT)}^{BMO^w(\TT)}\subseteq \overline{VMO^w(\TT)}^{BMO^w(\TT)}=VMO^w(\TT)\;.\]
For  the other direction , let  $f\in VMO^w(\TT)$, and consider the sequence $\set{R_{\epsilon}(f)}_{\epsilon>0}$ of rotations of $f$ by angle $\epsilon$,  defined on $\TT$ as $R_{\epsilon}(f)(e^{i \theta})=f(e^{i(\theta-\epsilon)}); \theta \in \R$.
Then from Theorem 2.1 in  \cite{Girela1999}, we have that  for all $\epsilon>0, R_{\epsilon}(f)\in C(\TT)$ and  $\ds \lim_{\epsilon \to 0} \norm{R_{\epsilon}(f)-f}_{BMO(\TT)}=0$. Since $BMO (\TT)\subseteq BMO^w(\TT)$, we also have that $\ds \lim_{\epsilon \to 0} \norm{R_{\epsilon}(f)-f}_{BMO^w(
\TT)}=0$. That is, $f\in \overline{C(\TT)}^{BMO^w(\TT)}. $

\end{proof}

%%%%%%%%%%%%%%%%%%%%%%%%%%%%%%%%%%%%%%%%%%%%%%%%%%%%%%%%%%%%%%%%
%%%%%%%%%%%%%%%%%%%%%%%%%%%%%%%%%%%%%%%%%%%%%%%%%%%%%%%%%%%%%%%%
%%%%%%%%%%%%%%%%%%%%%%%%%%%%%%%%%%%%%%%%%%%%%%%%%%%%%%%%%%%%%%%%
\section*{Acknowledgment} 
%The author would like to extend his thanks to the two anonymous referees for their comments and suggestions which helped greatly improve the manuscript.\\
This material is based upon work supported by the National Science Foundation under Grant No. DMS-1440140, National Security Agency under Grant No. H98230-20-1-0015, and the Sloan Grant under Grant No. G-2020-12602 while the
author participated in a program hosted by the Mathematical Sciences Research Institute in Berkeley, California, during the summer of 2020. 
\bibliography{WeakBMO}

\begin{thebibliography}{1}

\bibitem{Desouza1989}
S.~Bloom and G.~De~Souza.
\newblock {\it Atomic decomposition of generalized Lipschitz spaces}.
\newblock {\em Illinois Journal of Mathematics}, pages 682--686, 1989.

\bibitem{Coifman1974}
R.~Coifman.
\newblock {\it A real variable characterization of $h^p$}.
\newblock {\em Studia Mathematica}, {\bf 51}:269--274, 1974.

\bibitem{DeSouza1980}
G.~De~Souza.
\newblock {\em Spaces formed by special atoms}.
\newblock PhD thesis, SUNY at Albany, 1980.

\bibitem{Fefferman1971}
C.~Fefferman.
\newblock Characterization of bounded mean oscillation.
\newblock {\em Bulletin of the American Mathematical Society}, 77(4):587--588,
  1971.

\bibitem{Gill2017}
T.~Gill.
\newblock Banach spaces for the schwartz distributions.
\newblock {\em Real Analysis and Exchange}, 43(1):1--36, 2017.

\bibitem{Girela1999}
D.~Girela.
\newblock Analytic functions of bounded mean oscillation.
\newblock {\em Complex Functions Spaces (Merkrij\"{a}rvi 1999)}, Univ. Joensuu
  Dept. Math. Rep. Ser.(4):61--170, 2001.

\bibitem{Kwessi2020}
E.~Kwessi and G.~De~Souza.
\newblock Analytic characterization of high dimension weighted special atom
  spaces,.
\newblock {\em Complex Variables and Elliptic Equations}, 2020.

\bibitem{Kwessi2019}
E.~Kwessi, G.~De~Souza, D.~Ngalla, and N.~Mariama.
\newblock The special atom space in higher dimensions.
\newblock {\em Demonstratio Mathematica}, 33:131--151, 2020.

\bibitem{Zygmund2002}
A.~Zygmund.
\newblock {\em Trigonometric Series, Vol. I, II}.
\newblock Cambridge Mathematical Library, 2002.

\end{thebibliography}
\end{document}